\documentclass{ws-jktr}

\usepackage{graphics} 

\begin{document}

%\markboth{Authors' Names}
%{Instructions for Typesetting Camera-Ready Manuscripts}

%%%%%%%%%%%%%%%%%%%%% Publisher's Area please ignore %%%%%%%%%%%%%%
\catchline{}{}{}{}{}
%%%%%%%%%%%%%%%%%%%%%%%%%%%%%%%%%%%%%%%%%%%%%%%%%%%%%%%%%%%%%%%%%%%

\title{On the Kawauchi conjecture about the Conway polynomial of achiral knots}

\author{Nicola Ermotti, Cam Van Quach Hongler and Claude Weber}

\address{Section de math\' ematiques. Universit\' e de Gen\`eve
\\
 CP 64, CH-1211 Gen\`eve 4 (Switzerland)
 \\  
 nicola.ermotti@unige.ch
 \\
cam.quach@unige.ch
\\
claude.weber@unige.ch}

\maketitle
\begin{abstract}

We give a counterexample to the Kawauchi conjecture on the Conway polynomial of achiral knots which asserts that the Conway polynomial $C(z)$ of an achiral knot satisfies the splitting property $C(z)=F(z)F(-z)$ for a polynomial $F(z)$ with integer coefficients. We show that the Bonahon-Siebenmann decomposition of an achiral and alternating knot is reflected in the Conway polynomial. More explicitly, the Kawauchi conjecture is true for quasi-arborescent knots and counterexamples in the class of alternating knots must be quasi-polyhedral.

\end{abstract}



\section{Introduction}
	%) A SECTION HEADING
Knot polynomials like the Jones polynomial or the HOMFLY polynomial are very good at detecting chirality. As the Conway polynomials of a knot and its mirror image are the same, at first glance the Conway polynomial seems to be less powerful. Nonetheless, the Conway polynomial was known to satisfy some interesting properties : If a knot is achiral, then its determinant is a sum of squares (\cite {G}). Moreover, if we assume that the achiral knot $K$ is alternating then the leading coefficient $c_0(K)$ is, up to sign, a square of an integer (\cite {QW1}, \cite S).  Kawauchi (\cite {Ka})  conjectured for arbitrary achiral knots :
 \begin {conjecture}[Kawauchi]The Conway polynomial $C(z)$ of an achiral knot has the splitting property, i.e. $C(z)=F(z)F(-z)$ for a polynomial $F(z)$ with integer coefficients.
 \end {conjecture}
 Hartley and Kawauchi established in \cite{HK} this conjecture for strongly $-/+$achiral knots and for two-bridged knots. Later Hartley (\cite {H}) showed that the conjecture holds for arbitrary $-$achiral knots. The Kawauchi conjecture was also verified for all prime achiral knots of  $\leq$16 crossings (see Remark 4.3 in \cite {S}) and for a large class of hyperbolic achiral knots  tested in \cite {Co}. Despite these results, we show that the Kawauchi conjecture is not true.
  \\
In Section 2, we give a counterexample to the Kawauchi conjecture. This counterexample is an alternating $+$achiral knot of order 4 (Theorem 2.2).
\\
In Section 3, we will show how the Conway polynomial of an alternating achiral knot can be related to the decomposition of the knot into its arborescent and polyhedral parts. For this purpose we will use the decomposition of a knot projection into jewels and twisted band diagrams by Haseman (Conway) circles, as described by Bonahon-Siebenmann. In the appendix we recall the main lines of that decomposition which is carried out in detail in  (\cite {BS} and  \cite{QW2}): To each alternating knot $K$ we associate its  {\bf structure tree} $A(K)$: If $K$ is achiral, the symmetry induces an automorphism $\Phi$ on its structure tree $A(K)$ which has exactly one fixed point (\cite{EQW}).

\begin{definition} An achiral alternating knot $K$ is {\bf quasi-polyhedral} if the fixed point of $\Phi $ is either a jewel or a Haseman circle adjacent to two jewels.
The achiral alternating knot $K$ is {\bf quasi-arborescent} if the fixed point of $\Phi $ is a Haseman circle adjacent to two twisted band diagrams.
\end {definition}
We show that achiral quasi-arborescent knots always satisfy the splitting property (Theorem 3.3). Hence counterexamples such as the one given in Section 2 belong to the class of quasi-polyhedral knots. More generally one can be interested in non-necessarily alternating knots and raise the question: what is among such knots, the nature of counterexamples to Kawauchi conjecture.

Let us recall the definitions of $+/-$ achirality of knots.
\begin{definition} A knot $K$ is {\it achiral} if there is a diffeomorphism $h$ of $S^3$ such that $h(K)=K$ and $h$ reverses the orientation of $S^3$.
\\$K$ is {\it positively achiral} if $h$ preserves the orientation of $K$. It is  {\it negatively achiral} if $h$ reverses the orientation of $K$.

\section{A counterexample}
\end {definition}
\subsection{Even polynomials}
Define a polynomial $P(z)$ with coefficients in a unique factorization domain $A$ to be even if $P(z) = P(-z)$. Assume that char$A$ $\neq 2$. To be even is equivalent to having only monomials of even order.
Let $P(z) \in A[z]$ be even. If we substitute $-z$ for $z$ in its decomposition in irreducible elements, we see that this decomposition must be (up to a unit of $A$)  as follows:
\[
P(z)= ( \Pi_i p_i^{a_i}(z))  ( \Pi_k q_k(z)^ {b_k})  ( \Pi_k q_k(-z)^ {b_k})
\]
where:
\\ 
1) the polynomials $p_i(z)$ and $q_k(z)$ are irreducible;
\\
2) $p_i \neq p_j$  if $ i\neq j $ and
$ q_k \neq q_l$ if $ k\neq l$;
\\
3) the polynomials $p_i(z)$ are even and $q_k(z)$  non even.
\\
It is known that the Conway polynomial of a knot $C(z) \in \bf{Z}$ is even with constant term equal to $1$ (\cite {K}).  Hence one has the following proposition.

\begin {proposition} The Conway polynomial $C(z)$ of a knot satisfies the Kawauchi conjecture if and only if all the exponents $a_i$ in its decomposition into irreducible factors are even. 
 \end{proposition}

\subsection{A counterexample}

%\vspace{2in}
%\\FIGURE 1

\begin{figure}
\begin{center}
\includegraphics[scale=0.6]{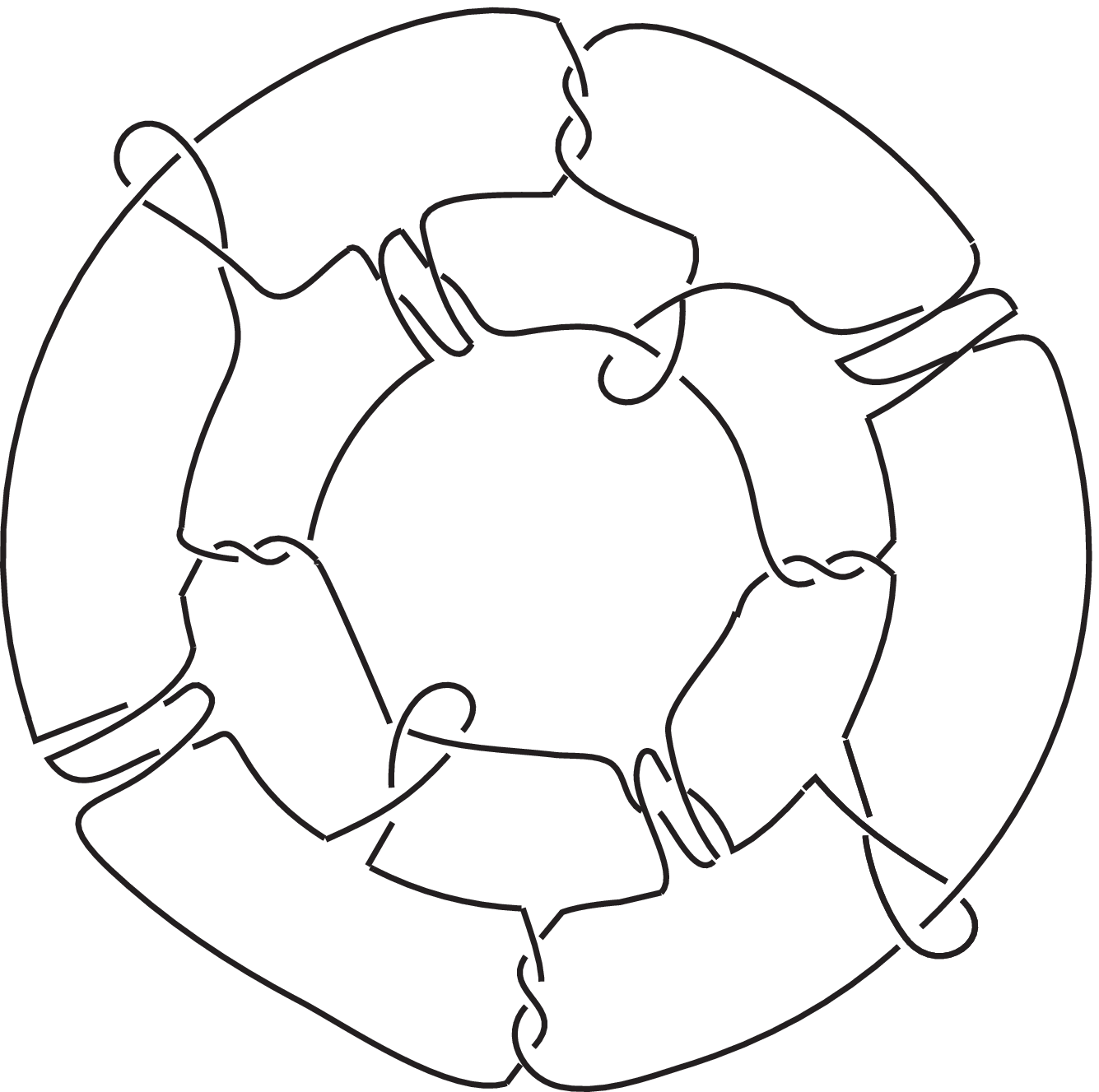}
\caption{}
\end{center}
\label{Figure 1}
\end{figure}

\begin {theorem} There exists an alternating achiral knot $K$ in $S^3$ such that its Conway polynomial is equal to:
\[
C_K (z)=P(z)(1+ z)(1- z)(2z^4 - 1)^2, 
\]
where 
\[
P(z)=4z^8 + 16z^6 + 12z^4 - 16z^2 + 1   
\]
and $P(z)$ is irreducible over $\bf{Z}$.
\end {theorem}
Once the irreducibility of $P(z)$ is established, the knot $K$ is clearly a counterexample to the Kawauchi conjecture.
\begin{proof}

1) Let us consider the alternating knot $K$ pictured in Figure 1. It is rather easy to realize that $K$ is positively achiral (the alternating projection of $K$ is invariant by a rotatory reflection of order 4). By KnotScape, its Conway polynomial is:

\[
C_K (z)=(4z^8 + 16z^6 + 12z^4 -16z^2 + 1)(1+z)(1- z)(2z^4 - 1)^2.  
\]
2) It remains to show that the polynomial $P(z)=4z^8 + 16z^6 +12z^4 -16z^2 + 1$ is irreducible over $\bf{Z}$. 
\\
First we decompose $P(z)$ over the field $\bf{F_3}$. This decomposition gives information  on the possible irreducible factors of $P(z)$, from which we can deduce that $P(z)$ is indeed irreducible over $\bf{Z}$.
On $\bf{F_3}$ the polynomial $P(z)$ is equal to:
\[
\overline{P}(z)=z^8 + z^6 + 2z^2 + 1. 
\]
It is easy to see that over  $\bf{F_3}$:

\[
\overline{P}(z)=\overline{q}(z) \overline{q}(-z)
\]
\\
where  $\overline{q}(z)=z^4 + z^3 + z^2 + 1$ is irreducible over $\bf{F_3}$.
\\
Hence $P(z)$ either is irreducible or decomposes as a product $q(z)q(-z)$ where $q(z)$ is irreducible non even and with reduction modulo 3 equal to $\overline{q}(z)$. We show that the latter case does not occur.
\\
Write $q(z) = (2z^4 + az^2-1)+(bz^3 + cz)$.
The comparison of $q(z)q(-z)$ with $P(z)$ produces the following equations for $a, b, c$:
\begin{align}
4a-b^2 &=16  \\       
a^2-2bc &=16  \\       
2a+c^2 &=16  
\end{align}
A little computation shows that the coefficients $a, b, c$ are all divisible by 4. Write $a=4a'$, $b=4b'$, $c=4c'$.
\\
Then the equations (2.1), (2.2) and (2.3) give rise to the following equations:
\begin{align}
a'-b'^2&=1\\
a'^2-2b'c'&=1\\
a'+2c'&=2
\end{align}
\\
The equation (2.6) implies that $a'$ is even and hence this contradicts (2.5).
\end{proof}

\begin {remark} The above counterexample satisfies the following conjecture of J.Conant (\cite {Co}).
\begin {conjecture}[Conant] Let $K$ be an achiral knot and $C(z)$ its Conway polynomial. Then there exists a polynomial  $F(Z) \in \bf{Z}_4[z^2]$ such that:
\[
           F(z)^2 = C(z)C(z^2)C(iz) \in \bf{Z}_4[z^2]
 \]          
\end {conjecture}
\end {remark}

\section{The Conway polynomial of achiral alternating knots}

Let us briefly recall some relations between the Conway polynomial and some topological properties of knots.
\\
1) Arborescence is not detected by the Conway polynomial. Let us recall that 
if $K$ is a knot then its Conway polynomial is an even polynomial with constant term equal to $1$. Under these constraints, every polynomial can be realized as the Conway polynomial of a knot, or even that of an arborescent (and hyperbolic) knot.

\begin{figure}
\begin{center}
\includegraphics[scale=1]{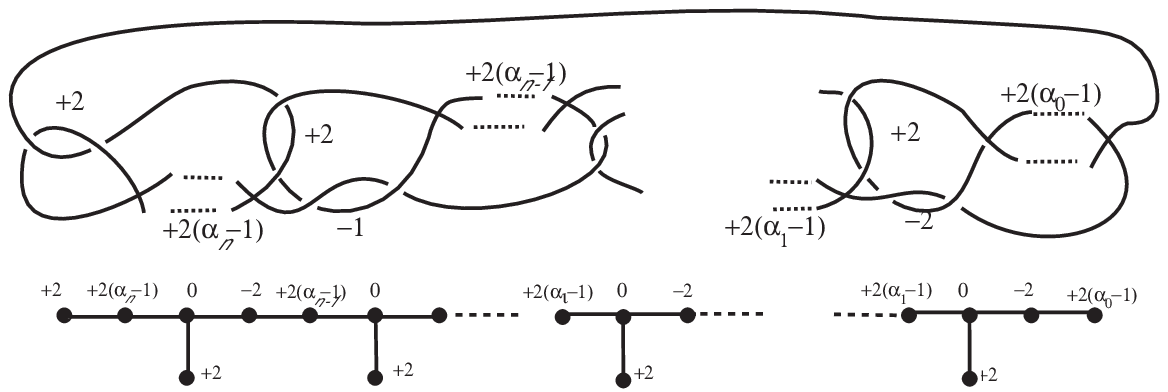}
\caption{}
\end{center}
\label{Figure 2t}
\end{figure}

\begin {theorem} (\cite {Q}). For a polynomial $C(z)= 1+ \sum_{i=1}^{n} a_i z^{2i}$, there exists an arborescent knot which admits $C(z)$ as its Conway polynomial.
\end {theorem}

\begin{proof} Let us consider the arborescent knot which has its weighted planar tree depicted in Figure 2. One can show that its Conway polynomial is equal to:
\[
C(z)=1+(\sum_{i=0}^{n}(-1)^{n+i}a_iz^{2(n-i+1)})
\]
\\
The proof is done by induction.
\end{proof}

2) The fact that a knot is alternating is reflected in the alternating form of its Alexander polynomial (another version of the Conway polynomial): all of its coefficients are non zero and alternate in signs.  However there exist alternating Alexander polynomials which are not associated to an alternating knot (\cite {LM})
\\
3) It is not known whether there is an achiral knot associated to each Conway polynomial satisfying the splitting property. However for the special case where $C(z)=F(z)^2$ and $F(z)$ is an even polynomial with $F(0)=1$, E. Flapan has shown that there exists a prime strongly +achiral knot which has $C(z)$ as Conway polynomial (\cite {F}).
\\
Let us now focus on the class of achiral alternating knots. Surprisingly the Conway polynomial proves to be more powerful than one would expect. For this purpose, we will use our definitions and results in (\cite {EQW}) (see also the appendix of this present paper).

\begin{theorem}Let $K$ be an achiral quasi-arborescent knot. Then there exists a polynomial $F(z)$ with integer coefficients such that $C(z)$ is equal to
\[
C(z) = F(z) F(-z)
\]
\end{theorem}
\begin{proof}If K is $-$achiral, by \cite {H} one has the result. If $K$ is $+$achiral and quasi-arborescent, then $K$ can be muted in a quasi-arborescent knot which is $-$achiral (in fact with a minimal projection of type $II$). As mutants have the same Conway polynomial (see for instance [\cite {Cr}), the theorem is proved.
\end{proof}
\begin{corollary}In the class of alternating knots, counterexamples to the Kawauchi conjecture are necessarily quasi-polyhedral.
\end{corollary}

\section{Remarks and Questions}
1) The Dasbach-Hougardy knot (\cite {DH}) is an arborescent knot which has a mutant which is $-$achiral as described by Theorem 5 in \cite  {EQW}. Although these two knots have the same Conway polynomial, they are distinguished by the fact that the Dasbach-Hougardy knot does not satisfy the Kauffman conjecture as shown in \cite {DH}.
\\
2)  Are $+$achiral alternating knots which satisfy the splitting property for their Conway polynomial quasi-arborescent?
\\
3) Can the arborescence and polyhedral structures in the class of alternating achiral knots be detected by other polynomial invariants ?

\section{Appendix: The structure tree and its automorphism}

Following Bonahon-Siebenmann, we decompose canonically in (\cite {QW2}) any link projection $\Pi$ which is connected and prime into diagrams called jewels and twisted band diagrams.

\begin{figure}[h]
\begin{center}
\begin{minipage}[b]{0.4\linewidth}
  \centering
  \includegraphics[scale=0.8]{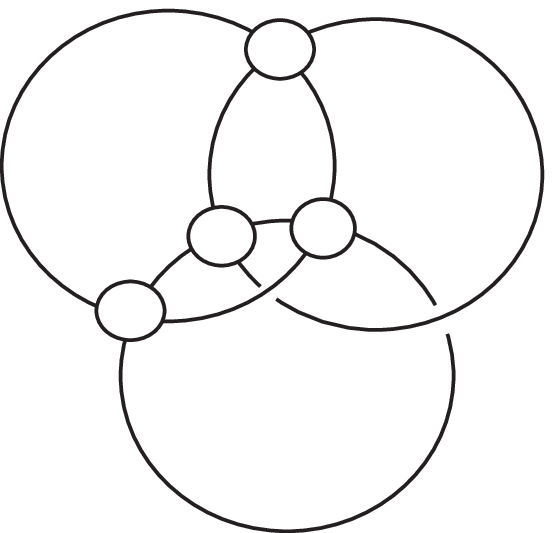}
  \caption{A jewel}
  \end{minipage}
\hspace{0.5cm}
\begin{minipage}[b]{0.4\linewidth}
  \centering
  \includegraphics[scale=0.3,bb=0 0 559 400]{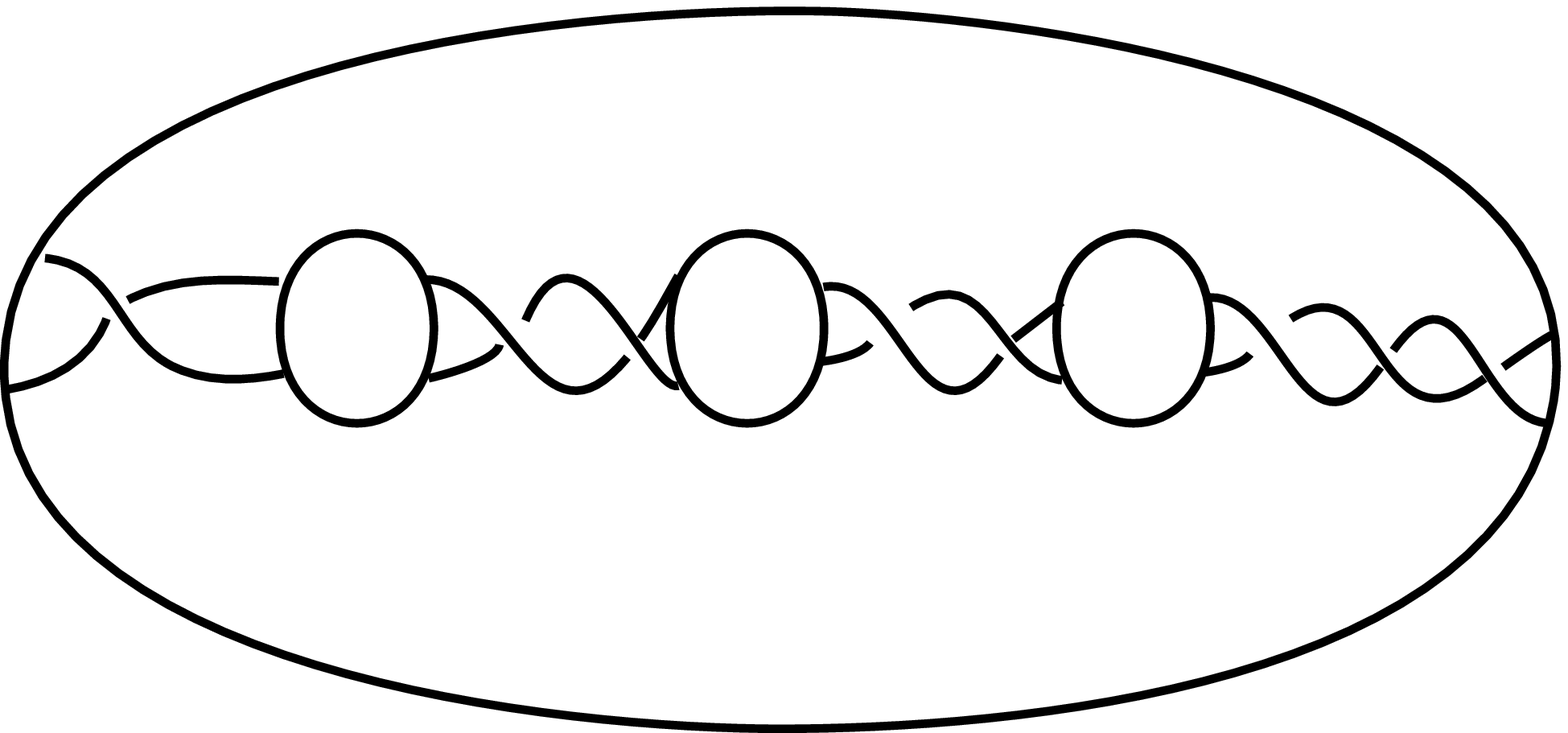}
    \caption{A twisted band diagram}
    \end{minipage}
    \end{center}
\end{figure}

The decomposition is performed along Haseman (Conway) circles, i.e. circles which intersect $\Pi$ transversely in four points and which satisfy an incompressibility condition. The decomposition of $\Pi$ is partially coded by a tree. Its vertices represent the diagrams of the decomposition and its edges the Haseman circles.
Any minimal projection of a prime alternating knot $K$ satisfies the required conditions. Moreover the tree depends only on $K$ and not on a particular minimal projection. We call it the {\bf structure tree} of $K$ and denote it by $A(K)$. An alternating knot is called {\bf arborescent} if all vertices of $A(K)$ are twisted band diagrams and is called {\bf polyhedral} if all vertices are jewels.
If $K$ is alternating and achiral, the symmetry induces an automorphism 
\[
\Phi:A(K)\rightarrow A(K).
\]
We prove in \cite {EQW} that:
\\
1) $\Phi$ has exactly one fixed point;
\\
2) the fixed point corresponds either to an invariant jewel or to a Haseman circle $\gamma$;
\\
3) in the latter case, $\gamma$ bounds either two jewels or two twisted band diagrams. These jewels or twisted band diagrams are exchanged by    $\Phi$.
\begin{figure}[h]
\begin{center}
\begin{minipage}[b]{0.4\linewidth}
  \centering
  \includegraphics[scale=0.3]{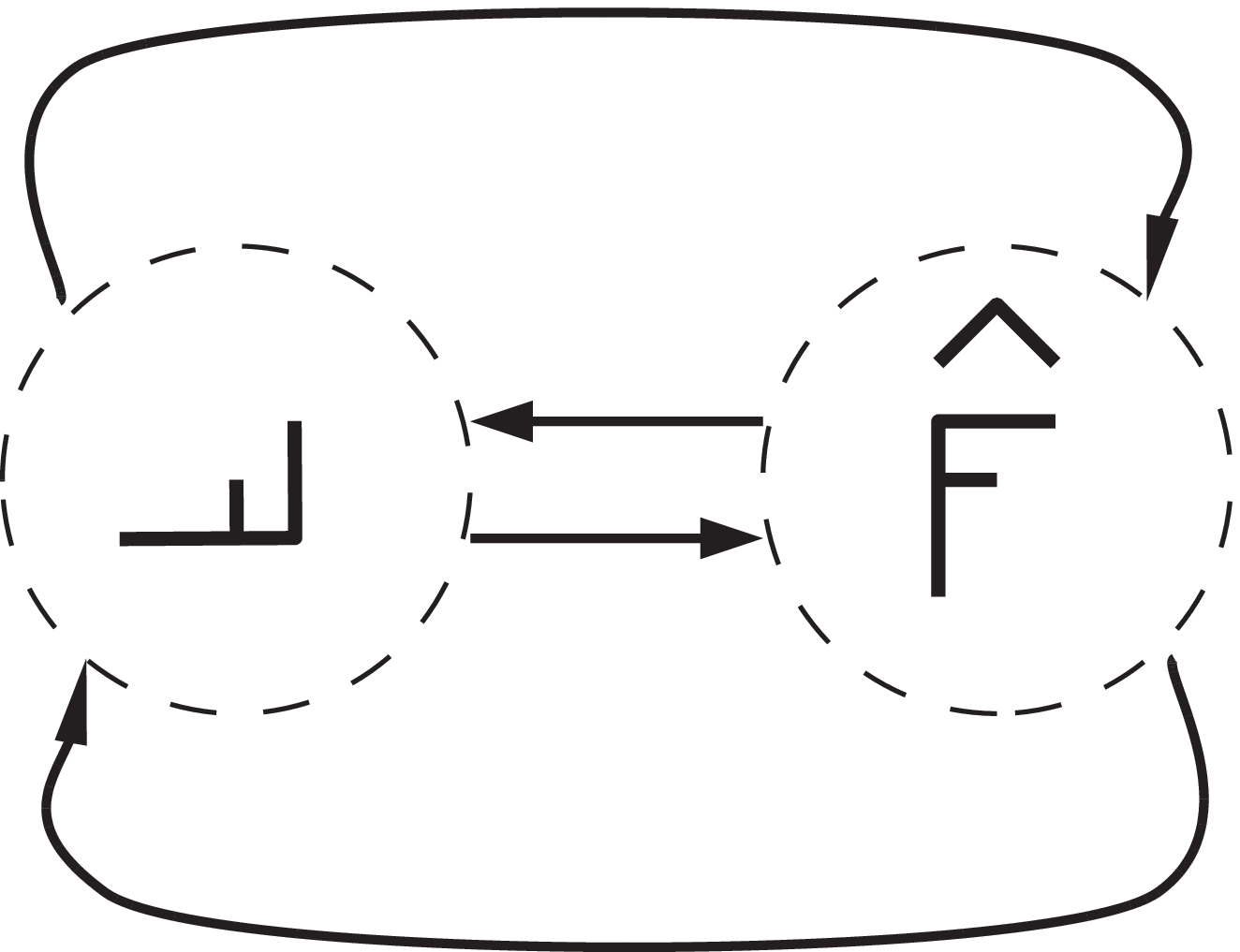}
  \caption{Projection of Type I}
  \end{minipage}
\hspace{0.5cm}
\begin{minipage}[b]{0.4\linewidth}
  \centering
  \includegraphics[scale=0.3]{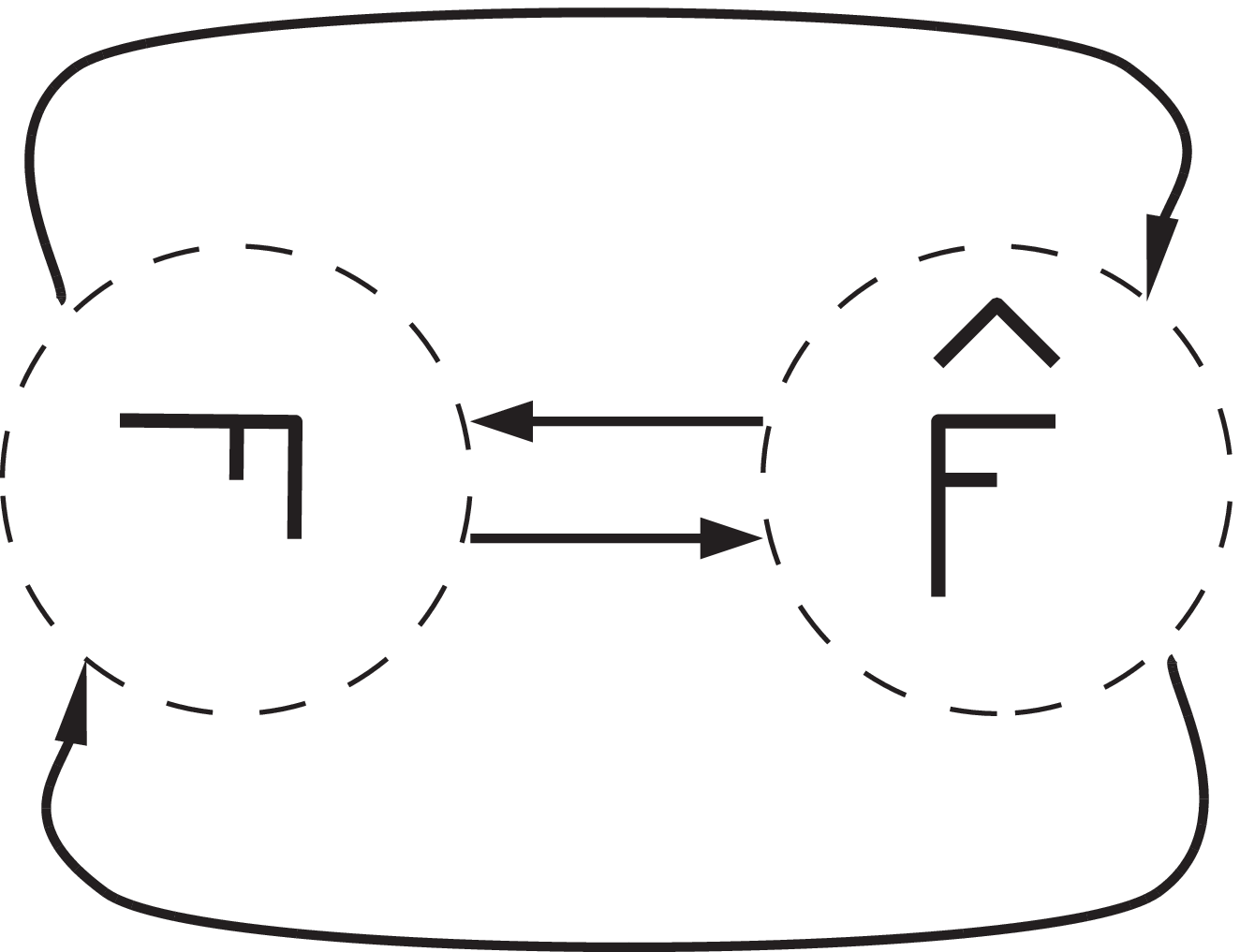}
    \caption{Projection of Type II}
    \end{minipage}
    \end{center}
\end{figure}

In the case 3), one has Theorem 5 in \cite {EQW}:
\begin{theorem}Let $K$ be an oriented $\pm$achiral knot. Suppose that $K$ has a projection with an invariant Haseman circle $\gamma$. Then up to a global change of orientation, $K$ admits a minimal projection of Type $I$ or Type $II$ as shown in Figures 5 and 6. 

\end {theorem}

\section*{Acknowledgments} 
We  thank Le Fonds National Suisse de la Recherche Scientifique for its  support.
\\
We thank Fran\c coise Michel for useful conversations.

\end{document}